\documentclass[11pt,english]{article}
\usepackage[T1]{fontenc}
\usepackage[latin9]{inputenc}
\usepackage{amsthm}
\usepackage{amsmath}
\usepackage{amssymb}

\makeatletter
  \theoremstyle{plain}
  \newtheorem*{conjecture*}{Conjecture}
\theoremstyle{plain}
\newtheorem{thm}{Theorem}
  \theoremstyle{plain}
  \newtheorem{lem}[thm]{Lemma}
  \theoremstyle{plain}
  \newtheorem{prop}[thm]{Proposition}

\usepackage{mathrsfs}\usepackage{amsthm}\usepackage{amscd}

\usepackage{hyperref}
\usepackage{graphics}
\usepackage{a4wide}
\@ifundefined{definecolor}
 {\usepackage{color}}{}

\newcounter{EQNR}
\renewcommand{\contentsname}{Contents\\
{\footnotesize\normalfont(A table of contents should normally not
be included)}}

\makeatother

\usepackage{babel}

\begin{document}

\title{Dynamics of Hilbert nonexpansive maps}

\author{Anders Karlsson%
\thanks{Supported in part by the Swiss NSF grant 200021 132528/1.%
}}
\maketitle
\begin{abstract}
In his work on the foundations of geometry, Hilbert observed that
a formula which appeared in works by Beltrami, Cayley, and Klein,
gives rise to a complete metric on any bounded convex domain. Some
decades later, Garrett Birkhoff and Hans Samelson noted that this
metric has interesting applications, when considering certain maps
of convex cones that contract the metric. Such situations have since
arisen in many contexts, pure and applied, and could be called nonlinear
Perron-Frobenius theory. This note centers around one dynamical aspect
of this theory.
\end{abstract}

\section{Introduction}

In a letter to Klein, published in \emph{Mathematisches Annalen} in
1895, Hilbert noted that a formula which Klein observed giving the
projective model of the hyperbolic plane, provides a metric on any
bounded convex domain. The formula is in terms of the cross-ratio
and appeared earlier in works of Cayley and Beltrami, although not
in this general context. Hilbert's concern lied in the foundation
of geometry, going back to the famous Euclid's postulates, and Hilbert's
fourth problem asking about geometries where straight lines are geodesics.
Nowadays, with an extensive development of metric geometry, we consider
these Hilbert geometries to be beautiful concrete examples of metric
spaces interpolating between (certain) Banach spaces and the real
hyperbolic spaces. To a certain limited extent it could also be viewed
as a simpler analogue to Kobayashi's (pseudo) metric in several variable
complex analysis, such as the Teichmüller metric.

It is then all the more remarkable that this metric discovered and
discussed from a foundational geometric point of view (by Hilbert
and Busemann) is a highly useful metric for other mathematical sciences,
providing yet another example of research driven by purely theoretical
curiosity leading to unexpected application. This was probably first
noticed, seemingly independently and simultaneously, by G. Birkhoff
\cite{Bi57} and H. Samelson \cite{Sa57}. The Hilbert metric is defined
via the cross-ratio from projective geometry. By virtue of being invariant
it is therefore not surprising that it is useful in the study of groups
of projective automorphisms, see for example \cite{Q10}. A philosophical
explanation for the abundance of applications beyond group theory
is that positivity (say of population sizes in biology, concentrations
in chemistry or prices in economics) is ubiquitous and that many even
nonlinear maps contract Hilbert distances. The contraction property
leads to fixed point theorems \index{fixed point theorem}which in
turn are a main source of existence of solutions to equations. Recently
a very good book has been published on this subject, nonlinear Perron-Frobenius
theory\index{Perron-Frobenius theory}, written by Lemmens and Nussbaum
\cite{LeN12}.

The following conjecture was made independently by Nussbaum and this
author about a decade ago:
\begin{conjecture*}
Let $C$ be a bounded convex domain in $\mathbb{R\mathrm{^{N}}}$
equipped with Hilbert's metric $d$. Let $f:C\rightarrow C$ be a
1-Lipschitz map. Then either there is a fixed point of $f$ in $C$
or the forward orbits $f^{n}(x)$ accumulate at one unique closed
face. 
\end{conjecture*}
In view of Nussbaum \cite{N88} we know the first part, namely if
the orbits do not accumulate on the boundary of $C$ then indeed there
is a fixed point. That the accumulation set do not depend much on
which point $x$ in $C$ that is iterated is not difficult to see,
but what is not quite settled yet is whether the limit set necessarily
belongs to the closure of only one face. In this brief text I will
focus on this issue and explain the known partial results. Taken together
they provide convincing support for the conjecture as they in particular
treat two basic and opposite situations, the strictly convex case
(Beardon) and the polyhedral case (Lins). A general result of Noskov
and the author shows that the conjecture in general is at least almost
true, see below.

As Volker Metz originally indicated to me, in practice it is often
hard to show that orbits are bounded (needed for establishing the
existence of a fixed point), but arguing by contradiction assuming
instead that the orbits are unbounded, then if one can show the existence
of a limiting face, this sometimes leads to the desired contradiction.
There are several studies on the bonded orbit case, which is not considered
here, for that we refer just to \cite{Le11,LeW11,LeN12,LeN13} and
references therein.

There is a wealth of applications of the Hilbert metric, see \cite{N88,N07,LeN12},
let me just add two papers, one on decay of correlation \cite{L95}
and another on growth of groups \cite{BE11}. A final remark is that
from a certain point of view, the conjecture is somewhat reminiscent
of the notorious invariant subspace problem in functional analysis,
although surely much less difficult.

\subsection*{Acknowledgements.}

I have at various times benefited from discussions on this general
topic with Alan Beardon, Gena Noskov, Volker Metz, Thomas Foertsch,
Roger Nussbaum, Brian Lins, Bruno Colbois, Constantin Vernicos, Enrico
Le Donne, and Bas Lemmens.

\section{Basic numbers associated to semicontractions}

Let $(X,d)$ be a metric space. A \emph{semicontraction} \index{semicontraction}$f:X\rightarrow X$
is a map such that \[
d(f(x),f(y))\leq d(x,y)\]
 for every $x,y\in X$. Other possible terms for this are 1-Lipschitz,
non-expanding or nonexpansive maps \index{nonexpansive map}. Since
$f$ is a self-map we may study it by iterating it.

 If one orbit $\{f^{n}(x)\}_{n>0}$ is bounded, then every orbit is bounded since \[ d(f^{n}(x),f^{n}(y))\leq d(x,y) \] for every $n$. Therefore having bounded or unbounded orbits is a well defined notion. We use the same terminology for subsets, semigroups or groups of semicontractions. For example, if a semigroup of semicontractions fixes a point, then trivially any orbit is bounded.
The \emph{translation number of}\index{translation number} $f$ is 
\[ \tau_{f}=\lim_{n\rightarrow\infty}\frac{1}{n}d(x,f^{n}(x)), \] which exists (and is independent of $x$) by the following lemma:
\begin{lem}
Let $a_{n}$ be a subadditive sequence of real numbers, that is $a_{n+m}\leq a_{n}+a_{m}.$ Then the following limit exists \[ \lim_{n\rightarrow\infty}\frac{1}{n}a_{n}=\inf_{m>0}\frac{1}{m}a_{m} \in\mathbb{R\cup\{-\infty\}}. \]\end{lem}
\begin{proof}
Given $\varepsilon>0,$ pick $M$ such that $a_{M}/M\leq\inf a_{n}/n$ +$\varepsilon.$ Decompose $n=k_{n}M+r_{n},$ where $0\leq r_{n}<M.$ Hence $k_{n}/n$ $\rightarrow1/M.$ Using the subadditivity and considering $n$ big enough ($n>N(\varepsilon))$ \begin{align*} \inf_{m>0}\frac1ma_{m}  &  \leq\frac1na_{n}\leq\frac1na_{k_{n}M+r_{n}} \leq\frac1n(k_{n}a_{M}+a_{r_{n}})\\ &  \leq\frac1Ma_{M}+\varepsilon\leq\inf_{m>0}\frac1ma_{m}+2\varepsilon. \end{align*} Since $\varepsilon$ is at our disposal, the lemma is proved.
\end{proof}

Furthermore we define \[ \delta_{f}(x)=\lim_{n\rightarrow\infty}d(f^{n}(x),f^{n+1}(x)), \] which exists because it is the limit of a bounded monotone sequence. Finally, let \[ D_{f}=\inf_{x\in X}d(x,f(x)) \] be \emph{the minimal displacement}, sometimes also called the translation length.
Since \[ d(x,f^{n}(x))\leq d(x,f(x))+...+d(f^{n-1}(x),f^{n}(x))\leq nd(x,f(x)) \] \newline we have \[ 0\leq\tau_{f}\leq D_{f}\leq\delta_{f}(x). \] If $f$ is an isometry, then $\delta_{f}$ is constant along every orbit, which means that typically $\delta_{f}(x)>D_{f}$. Moreover, it might happen that $0=\tau_{f}<D_{f}$, for example if $f$ has finite order and $X$ is an orbit.

A metric space is \emph{proper }if every closed and bounded subset
is compact. A theorem of Calka \cite{Ca84} asserts that provided
that $X$ is proper, if an orbit is unbounded, then it actually escapes
every bounded set, so that $d(f^{n}x,x)\rightarrow\infty$ as $n\rightarrow\infty$
for any $x$. This applies to the Hilbert metric case that we consider
here. In this case this fact was independently discovered by Nussbaum
\cite{N88}. For non-proper spaces this is not true in general.

\section{Horofunctions\index{horofunction}}

Horofunctions and horospheres first appeared in noneuclidean geometry and complex analysis in the unit disk or upper half plane. Consider the unit disk $D$ in the complex plane. The metric with constant Gaussian curvature $-1$ is 
$$ ds=\frac{2\left\vert dz\right\vert }{1-\left\vert z\right\vert ^{2}}\,\,\text{ or }\,\,d(0,z)=\log\frac{1+\left\vert z\right\vert }{1-\left\vert z\right\vert }. $$

It is called the Poincar\'{e} metric and for a point $\zeta\in\partial D$ one has the horofunction
$$h_{\zeta}(z)=\log\frac{\left\vert \zeta-z\right\vert ^{2}}{1-\left\vert z\right\vert ^{2}}. $$ 
These appear explicitly or implicitly in for example the Poisson formula in complex analysis, Eisenstein series, and the Wolff-Denjoy theorem.
An abstract definition was later introduced by Busemann who defines the \emph{Busemann function} associated to a geodesic ray $\gamma$ to be the function 
$$ h_{\gamma}(z)=\lim_{t\rightarrow\infty}d(\gamma(t),z)-t. $$ 
Note here that the limit indeed exists in any metric space, since the triangle inequality implies that the expression on the right is monotonically decreasing and bounded from below by $-d(z,\gamma(0))$. The convergence is moreover uniform if $X$ is proper as can be seen from a 3$\varepsilon$-proof using the compactness of closed balls.

Horoballs are sublevel sets of horofunctions $h(\cdot)\leq C$. In euclidean geometry horoballs are halfspaces and in the Poincaré disk model of the hyperbolic plane horoballs, or horodisks in this case, are euclidean disks tangent to the boundary circle.

There is a more general definition probably first considered by Gromov around 1980. Namely, let $X$ be a complete metric space and let $C(X)$ denote the space of continuous real functions on $X$ equipped with topology of uniform convergence on bounded subsets. Fix a base point $x_{0}\in X$. Consider now the map $\Phi:X\rightarrow C(X)$ defined by \[ \Phi:z\mapsto d(z,\cdot)-d(z,x_{0}). \] (A related map was considered by Kuratowski \cite{K48} and independently K. Kunugui \cite{Ku35} in the 1930s.) We will sometimes denote by $h_{z}$ the function $\Phi(z).$ Note that every $h_{z}$ is 1-Lipschitz because 
$$ \left\vert h_{z}(x)-h_{z}(y)\right\vert =\left\vert d(z,x)-d(z,x_{0} )-d(z,y)+d(z,x_{0})\right\vert \leq d(x,y) $$ which applied to $y=x_{0}$ gives 
$$ \left\vert h_{z}(x)\right\vert \leq d(x,x_{0}). $$ We have that the map $\Phi$ is a continuous injection.

We denote the closure $\overline{\Phi(X\text{)}}$ by $\overline{X}^{H}$
or $X\cup\partial_{H}X$ and call it the horofunction compactification
of $X$. The elements in $\partial_{H}X$ are called horofunctions.
In a sense, horofunctions are to metric spaces what linear functionals
(of norm 1) are to vector spaces.

\section{Horofunctions and semicontractions}

It has long been observed, starting with the studies of Wolff and
Denjoy on holomorphic self-maps, that horofunctions is a relevant
notion for the dynamical behaviour of semicontractions.

\subsection{An argument of Beardon}

Here is a sketch of a nice argument due to Beardon \cite{Be97}. Suppose
that the semicontraction $f$ of a proper metric space $X,$ can be
approximated by uniform contractions $f_{k}$, so that for every $x$
one has $f_{k}(x)\rightarrow f(x).$ By the contraction mapping principle
every $f_{k}$ has a unique fixed point $y_{k}$ since $X$ is complete.
Now take a limit point $y$ of this sequence in a compactification
$\overline{X}$ of $X$. For simplicity of notation we assume that
$y_{k}\rightarrow y$. If the limit point belongs to $X$, then it
must be fixed by $f$, indeed: \[
f(y)=\lim_{k\rightarrow\infty}f_{k}(y_{k})=\lim_{k\rightarrow\infty}y_{k}=y.\]
 If the limit point instead belongs to the associated ideal boundary
$\partial X=\overline{X}\setminus X$, then we define the associated
{}``horoballs'': $z\in H_{\left\{ y_{k}\right\} }(x)$ if and only
if $z$ is the limit of points $z_{k}$ belonging to the closed balls
centered at $y_{k}$ with radius $d(x,y_{k}).$ These are invariant
in the sense that $f(H_{\left\{ y_{k}\right\} }(x))\subset H_{\left\{ y_{k}\right\} }(x)$
as is straightforward to verify (\cite{Be97}).

In the special case that the compactification is the horofunction
compactification and the ideal boundary $\partial_{H}X$, the horoballs
are not dependent on the sequence, just the limit $y$ and one has
for any $x\in X$\[
h_{y}(f(x))=\lim_{k\rightarrow\infty}d(y_{k},f(x))-d(y_{k},x_{0})=\lim_{k\rightarrow\infty}d(y_{k},f_{k}(x))-d(y_{k},x_{0})\]
\[
=\lim_{k\rightarrow\infty}d(f_{k}(y_{k}),f_{k}(x))-d(y_{k},x_{0})\leq\lim_{k\rightarrow\infty}d(y_{k},x)-d(y_{k},x_{0})=h_{y}(x).\]

This implies in particular :
\begin{thm}
(\cite{Be97}) Let $C$ be a bounded strictly convex domain in $\mathbb{R\mathrm{^{N}}}$
equipped with Hilbert's metric $d$. Let $f:C\rightarrow C$ be a
1-Lipschitz map. Then either there is a fixed point of $f$ in $C$
or the forward orbits $f^{n}(x)$ converge to a unique boundary point
$z\in\partial C$. 
\end{thm}
In a more recent paper of Gaubert and Vigeral \cite{GV12}, an interesting
sharpening (in some sense) of Beardon's argument is used to establish
a result inspired by the Collatz-Wielandt characterisation of the
Perron root in linear algebra.
\begin{thm}
(\cite{GV12}) Let $f$ be a semicontraction of a complete metrically
star-shaped hemi-metric space $(X,d)$. Then,\[
D_{f}=\tau_{f}=\max_{h\in\overline{X}^{H}}\inf_{x\in X}h(x)-h(f(x)).\]

\end{thm}
This is just the first part of the theorem, it has an interesting
second part as well. For precise definitions we refer to the paper
in question, but roughly the metric space is allowed to be asymmetric
and should have a canonical choice of geodesics which have Busemann
non-positive curvature property. If $(X,d)$ is not proper one has
to intepret the horofunctions appropriately. To a certain extent this
theorem unifies Beardon's result with the result in the next paragraph
in the setting it considers. Both Beardon and Gaubert-Vigeral results
apply to semicontractions of Hilbert's metric.

\subsection{A general abstract result}

Beardon's argument depended on finding approximations to $f$. Here
is a quite general version of the relationship between semicontractions
and horofunctions, not using any approximating sequence:
\begin{prop}
\label{pro:K}(\cite{K01}) Let $(X,d)$ be a proper metric space.
If $f$ is a semicontraction with translation number $\tau$, then
for any $x\in X$ there is a function $h\in\overline{X}^{H}$ such
that \[
h(f^{k}(x))\leq-\tau k\]
 for all $k>0.$\end{prop}
\begin{proof}
Given a sequence $\epsilon_{i}\searrow0$ we set $b_{i}(n)=d(f^{n}(x),x)-(\tau-\epsilon_{i})n.$
Since these numbers are unbounded, we can find a subsequence such
that $b_{i}(n_{i})>b_{i}(m)$ for any $m<n_{i}$ and by sequential
compactness we may moreover assume that $f^{n_{i}}(x)\rightarrow h\in\overline{X}^{H}$,
for some $h$. 

We then have for any $k\geq1$ that \begin{align*}
h(f^{k}(x)) & =\lim_{i\rightarrow\infty}d(f^{k}(x),f^{n_{i}}(x))-d(x,f^{n_{i}}(x))\\
 & \leq\liminf_{i\rightarrow\infty}d(x,f^{n_{i}-k}(x))-d(x,f^{n_{i}}(x))\\
 & \leq\liminf_{i\rightarrow\infty}b_{i}(n_{i}-k)+(\tau-\epsilon_{i})(n_{i}-k)-b_{i}(n_{i})-(\tau-\epsilon_{i})n_{i}\\
 & \leq\lim_{i\rightarrow\infty}-\tau k+\epsilon_{i}k=-\tau k,\\
\end{align*}
as was to be shown.
\end{proof}

\section{Horofunctions and asymptotic geometry of Hilbert geometries}

Being a convex domain in euclidean space $C$ has a natural extrinsic
boundary $\partial C$. One faces the challenge to compare it with
the intrinsic metric boundary $\partial_{H}X$. Two papers that consider
this comparison are \cite{KMN06} and \cite{W08}. In the former only
the case of simplices are treated and very precise description is
obtained. In the latter reference the general situation is considered
which is useful here. It is observed that the horoballs are convex
sets, and this means in particular that any intersection of horoballs
where the base points tends to the boundary of $X$ can at most contain
one closed face. Walsh also shows that any sequence convergent in
the horofunction boundary also converges to a point in $\partial C$.
We will make free use of these facts below. 

Notice that since any two orbits lie on a bounded distance from each
other, it follows from simple well-known estimates on the Hilbert
metric that the closed face, as predicted by the conjecture, must
be independent of the point $x$ being iterated. This follows from
the more general statement:
\begin{prop}
(\cite{KN01}) Let $C$ be a bounded convex domain with Hilbert's
metric $d$. Fix $y\in C$ and define the Gromov product\index{Gromov product}:\[
(x,x')=\frac{1}{2}(d(x,y)+d(x',y)-d(x,x')).\]
Assume that we have two sequences of points in $C$ that converges
to boundary points: $x_{n}\rightarrow\overline{x}\in\partial C$ and
$z_{n}\rightarrow\overline{z}\in\partial C$. If $\partial C$ does
not contain the line segment $\left[\overline{x},\overline{z}\right],$
then there is a constant $K=K(\overline{x},\overline{z})$ such that
\[
\limsup_{n\rightarrow\infty}(x_{n},z_{n})\leq K.\]

\end{prop}
As remarked in the same paper this is useful for the studies of iteration
of semicontractions $f$, see Theorem \ref{thm:KN} below. It has
also found other applications, such as in \cite{FK05}. To explain
the point above about different orbits in detail, notice first that\[
d(f^{n}(x),f^{n}(z))\leq d(x,z)\]
for all $n>0.$ Hence for any unbounded $f$ one has\[
(f^{n}(x),f^{n}(z))=\frac{1}{2}(d(f^{n}(x),y)+d(f^{n}(z),y)-d(f^{n}(x),f^{n}(z)))\rightarrow\infty.\]
In view of the proposition we have that any subsequence of the iterates
on $x$ and the same for $z$ must accumulate on the same closed faces.

\section{Some consequences for Hilbert metric semicontractions}

\subsection{Karlsson-Noskov}

In \cite{KN01} a certain simple asymptotic geometric fact in terms
of the Gromov product is established as recalled above. As a corollary
of this, in view of the arguments in \cite{K01}, we concluded that:
\begin{thm}
\label{thm:KN}(\cite{KN01}) Let $C$ be a bounded convex domain
in $\mathbb{R\mathrm{^{N}}}$ equipped with Hilbert's metric $d$.
Let $f:C\rightarrow C$ be a 1-Lipschitz map. Then either there is
a fixed point of $f$ in $C$ or there is a boundary point $z\in\partial C$
such that the line segment between $z$ and any limit point of a forward
iteration $f^{n}(x)$ is contained in $\partial C$. 
\end{thm}
In \cite{K05} the same result but in a more general setting is discussed.

\subsection{Monotone orbits}

If the sequence $d(x,f^{n}(x))$ grows monotonically, then every convergent
subsequence can play the role of $n_{i}$ as in the proof of Proposition
\ref{pro:K}. This means that in Theorem \ref{thm:KN} the special
property of $z$ is held by any limit point, thus verifying the conjecture.
Explicitly:
\begin{thm}
Let $C$ be a bounded convex domain in $\mathbb{R\mathrm{^{N}}}$
equipped with Hilbert's metric $d$. Let $f:C\rightarrow C$ be a
1-Lipschitz map. Suppose that for some $x\in X$, $d(x,f{}^{n}(x))\nearrow\infty$
monotonically. Then there is a closed face that contains every accummulation
of the iterations of $f$.
\end{thm}
The monotonicity can be weakend for this argument to go through with
some smaller addition, although this may not suffice for the general
case.

\subsection{Linear drift}

If $\tau_{f}$ is strictly positive, then by the general abstract
result, Proposition \ref{pro:K}, the orbit goes deeper and deeper
inside the horoballs of a fixed horofunction. Since these are convex
sets, their intersction is convex and must therefore be a closed face.
The orbit can only accumulate here since the horoballs contain all
orbit points:
\begin{thm}
Let $C$ be a bounded convex domain in $\mathbb{R\mathrm{^{N}}}$
equipped with Hilbert's metric $d$. Let $f:C\rightarrow C$ be a
1-Lipschitz map. Suppose that $\tau_{f}>0$, then there is a closed
face that contains every accummulation point of the iterations of
$f$.
\end{thm}
Here it is interesting to remark, as is done in \cite{LeN13}, that
the conjecture also holds true in the quite opposite situation that
$\liminf_{k\rightarrow\infty}d(f^{k}(x),f^{k+1}(x))=0$ as shown by
Nussbaum in \cite{N07}.

\subsection{Two-dimensional geometries}

In this note we have described two methods: Beardon's method with
approximations and another one coming from \cite{K01} in terms of
orbits. Unfortunately there is at present no known relation between
the associated boundary points these two arguments give. If they would
give the same point, then the conjecture would follow. In dimension
two where the extrinsic geometry is limited, one can play out these
two boundary points to conclude:
\begin{thm}
The Conjecture is true in dimension 2.\end{thm}
\begin{proof}
By Karlsson-Noskov we have that there is a star containing all orbit
points. At worst this is made up of two adjacent line segments, suppose
this is the case since otherwise we are done. Beardon's boundary point
must lie in one of these. But no matter in which, since the horoballs
are invariant sets we again get that since the orbit accumulates at
all the faces in the star we get a contradiction, and there can only
be one closed face containing all limit points. In the case Beardon's
point is the corner point, then this coincides with the point in \cite{K01}.
So Beardon's point is a limit point, but in this case the closed face
is just that same point. This proves the conjecture.
\end{proof}

\section{Theorems of Lins and Nussbaum}

Lins showed in his thesis \cite{Li07} that the conjecture is true
for every polyhedral domain. This is particularly interesting since
it seems to be the case most opposite to the one that Beardon handled,
i.e. the strictly convex domains. The proof goes via an isometric
embedding into a finite dimensional Banach space.

Lins and Nussbaum \cite{LiN06} showed that for projective linear
maps a stronger version of the conjecture is true: the orbits converge
to a finite number of points. (The paper \cite{FK05}, which proves
the convergence of geodesics rays, is somewhat analogous to this).
This constitutes a beautiful addendum to the classical Perron-Frobenius
theorem. Further established cases of the conjecture for specific
maps of practical interest can be found in \cite{AGLN06,LiN06,N07,LeN12}.

We conclude by remarking that, in contrast to the linear case of \cite{LiN06},
Lins \cite{Li07} has shown that in some sense the conjecture is best
possible: for any simplex and convex subset of the boundary, he constructs
a semicontraction whose limit set contains this boundary subset.

\noindent Anders Karlsson 

\noindent Section de mathématiques \\
Université de Genève

\noindent 2-4 Rue du Lièvre\\
Case Postale 64

\noindent 1211 Genève 4, 

\noindent Suisse 

\noindent e-mail: anders.karlsson@unige.ch 
\end{document}